\documentclass[10pt]{article}
\usepackage{mathrsfs}
\usepackage{amsfonts}
\usepackage{cite} 
\usepackage{url}  
\usepackage{ifthen}  
\usepackage{multicol}   
\usepackage[utf8]{inputenc} 
\usepackage{amsmath}
\usepackage{amsfonts}
\usepackage{mathrsfs}
\usepackage{amsmath,amssymb,amscd,bbm,amsthm,mathrsfs,dsfont}
\usepackage{bbm}
\usepackage{latexsym}
\numberwithin{equation}{section}
\newtheorem{theorem}{Theorem}[section]
\newtheorem{lem}{Lemma}[section]

\newtheorem{remark}{Remark}[section]

\newtheorem{cor}{Corollary}[section]
\def\includegraphics{}
\newboolean{publ}

\begin{document}
\title{\bf A partially observed non-zero sum differential game of forward-backward stochastic differential equations and its application in finance}
\author{
Jie Xiong $^{a}$, Shuaiqi Zhang $^{b}$, Yi Zhuang $^{c}$\\
\small{$^a$ Department of Mathematics, University of Macau, Macau, PR China}\\
\small{$^b$ School of Economics and Commerce, Guangdong University of Technology, Guangzhou 510520, PR China}\\
\small{$^{c}$ School of Mathematics, Shandong University, Jinan 250100, PR China}
}
\maketitle
\begin{abstract}
In this article, we concern a kind of partially observed non-zero sum stochastic differential game based on forward and backward stochastic differential equations (FBSDEs). It is required that each player has his own observation equation, and the corresponding open-loop Nash equilibrium control is required to adapted to the filtration that the observation process generated. To find this open-loop Nash equilibrium point, we prove the maximum principle as a necessary condition of the existence of this point, and give a verification theorem as a sufficient condition to verify it is the real open-loop Nash equilibrium point. Combined this with reality, a financial investment problem is raised. We can obtain the explicit observable investment strategy by using stochastic filtering theory and the results above.

\end{abstract}


\ \ \ \ \ \ \ \ \ \ \ \ \ \ \ \ \ information, stochastic filtering

\section{Introduction}
\subsection{Historical contribution}
	
The general theory of backward stochastic differential equation (BSDE) was first introduced by Pardoux and Peng \cite{Peng90}. For the BSDE coupled with a forward stochastic differential equation (SDE), it is  so-called the forward and backward stochastic differential equation (FBSDE), which has important applications in many areas in our society. In stochastic control area, the Hamiltonian system is one of the form of FBSDEs. More essentially in financial market, the famous Black-Scholes option pricing formula can be deduced by a certain FBSDE. Some research based on FBSDE is surveyed by Ma and Yong \cite{MaYong99}.

In stochastic control theory, one can use control to reach a maximum or minimum objection based on stochastic differential system. Peng \cite{Peng93} firstly considered the maximum principle of convex domain forward-backward stochastic control system. In the following, Xu \cite{Xu95} dealt with a case that control domain doesn't need to be convex and there is no control variable in diffusion coefficient in the forward equation. In more general case, Tang and Li \cite{TangLi94} considered that the control domain is non-convex and diffusion coefficient contains control variable. Moreover, Shi and Wu \cite{ShiWu07}, \cite{ShiWu12} solved the corresponding fully-coupled case, etc. All these previous work were based on the ``complete information'' case, meaning that the control variable is adapted to the truth complete filtration. In reality, there are many cases the controller can only obtains  ``partial information'', reflecting in mathematics that the control variable is adapted to the filtraion generated by an observable process. Based on this phenomenon, Xiong and Zhou \cite{XiongZhou07} dealt with a Mean-Variance problem in financial marcket that the investor's optimal portfolio is only based on the stock and bond process he observed. This assumption of partial information is indeed natural in financing market. What's more, Wang and Wu \cite{WangWu08} considered the Kalman-Bucy filtering equation of FBSDE system. Huang, Wang and Xiong \cite{HuangWangXiong09} dealt with the backward stochastic control system under partial information. Wang and Wu \cite{WangWu09}, Wu \cite{Wu10}, Wang Wu and Xiong \cite{WangWuXiong13} solved the partially observed case of forward and backward stochastic control system.

The game theory was firstly constructed by Von Neumann since 1928. Nash \cite{Nash50Equilibrium} - \cite{Nash53} made the fundamental contribution in the Non-cooperate Games, considered there are N-players acting independently to maximize their own objective conducted. He gave the notion of equilibrium point. Then, Isaacs \cite{Isaacs55}, Basar and Olsder \cite{BasarOlsder82} conducted the game research on differential equation system. Varaiya \cite{Varaiya76}, Eliott and Davis\cite{EliottDavis81} considered  the stochastic case. Next, many articles of forward stochastic differential games which is based on SDEs appeared, like Hamadene \cite{Hamadene95a} - \cite{Hamadene99}, Karoui and Hamadene \cite{KarouiHamadene03}, Wu \cite{Wu05}, {\O}ksendal \cite{TTKOksendal08}, etc. For the backward case, Yu and Ji \cite{YuJi08} studied the Linear Quadratic (LQ) system, Wang and Yu \cite{WangYu10} gave the maximum principle of backward system. \O ksendal and Sulem \cite{OksendalSulem12}, Hui and Xiao \cite{HuiXiao12} had a research on the maximum principle of forward-backward system. Recently, Tang and Meng \cite{TangMeng10} solved the partial information case of zero-sum forward and backward system. Wang and Yu \cite{WangYu12} solved the partial information case of non-zero sum backward system, etc.

In our article, we generate game theory to the partially observed  non-zero sum forward-backward system. The main difference here is that, we suppose every player has his own observation equation, not just as partial information that focusing only on a smaller sub-filtration. In section 1, we introduce some historical contributions and make notions we need. In section 2, we establish the necessary condition of maximum principle for Nash equilibrium point and give a sufficient condition (verification theorem) to help us check if the candidate equilibrium points are real. In section 3, we consider a reasonable financial investment problem and use the theorems in section 2 to obtain the open-loop Nash equilibrium point and give the explicit observable solution of investment strategy.

\subsection{Basic Notions}
Throughout our article, we denote $(\Omega,\mathcal{F},\{\mathcal{F}_{t}\}_{t\geq0},\mathbb{P})$ the complete probability space, on which $(W(\cdot),Y_1(\cdot)$,
$Y_2(\cdot))$
be a standard 3-dimensional $\mathcal{F}_{t}$ Brownian motion. Let $\mathcal{F}_{t}^W,\mathcal{F}_{t}^1,\mathcal{F}_{t}^2$ be the natural filtration generated by $W(\cdot),Y_1(\cdot),Y_2(\cdot)$ respectively. We set $\mathcal{F}_t=\mathcal{F}_{t}^W\otimes\mathcal{F}_{t}^1\otimes\mathcal{F}_{t}^2$. For fixed terminal time $T$, $\mathcal{F}=\mathcal{F}_T$. What's more, we denote the 1-dimensional Euclidean space by $\mathbb{R}$, the Euclidean norm by $|\cdot|$, and the transpose of matrix $A$ by $A^\tau$, the partial derivative of function $f(\cdot)$ with respect to $x$ by $f_x(\cdot)$. We also denote the $\mathbb{L}^2_\mathcal{F}(0,T;S)$ representing the set of $S$-valued, $\mathcal{F}_t$-adapted square integrable process (i.e. $\mathbb{E}\int_0^T|x(t)|^2dt<\infty$), and the $\mathbb{L}^2_\mathcal{F}(\Omega;S)$ representing the set of $S$-valued, $\mathcal{F}$-measured square integrable random variable. In the following discussion, we only consider 1-dimensional case if there is no specific illustration.

\subsection{Problem formulation}
We consider a partially observed stochastic differential game problem of forward-backward stochastic systems, focusing on necessity and sufficiency of the existence of open-loop Nash equilibrium point.

We formulate the controlled forward and backward stochastic differential equation (FBSDE) as

\begin{equation}\label{FBSDE}
\left\{
\begin{aligned}
dx(t)=&b(t,x(t),v_{1}(t),v_{2}(t))dt+\sigma(t,x(t),v_{1}(t),v_{2}(t))dW(t)\\
+&\sigma_{1}(t,x(t),v_{1}(t),v_{2}(t))dW_{1}^{v_1,v_2}(t)+\sigma_2(t,x(t),v_{1}(t),v_{2}(t))dW_2^{v_1,v_2}(t),\\
-dy(t)=&f(t,x(t),y(t),z(t),z_1(t),z_2(t),v_{1}(t),v_{2}(t))dt-z(t)dW(t)-z_1(t)dY_1(t)-z_2(t)dY_2(t),\\
x(0)=&x_0,\\
y(t)=&g(x(T)),
\end{aligned}
\right.
\end{equation}
where $v_1(\cdot),v_2(\cdot)$ are two control processes taking values in convex sets $U_1\subset\mathbb{R},U_2\subset\mathbb{R}$ respectively, $W_1^{v_1,v_2}(\cdot)$ and $W_1^{v_1,v_2}(\cdot)$ are controlled stochastic process taking values in $\mathbb{R}$, $b,\sigma,\sigma_1,\sigma_2: \Omega\times[0,T]\times\mathbb{R}\times U_1\times U_2\mapsto\mathbb{R}$, $f: \Omega\times[0,T]\times\mathbb{R}\times\mathbb{R}\times\mathbb{R}\times\mathbb{R}\times\mathbb{R}\times U_1\times U_2\mapsto\mathbb{R}$, $g:\Omega\times\mathbb{R}\mapsto\mathbb{R}$ are continuous maps, $x_0\in\mathbb{R}$, $g(x(T))$ is a $\mathcal{F}_T$ measurable square integrable random variable. Here for simplicity, we omit the notation of $\omega$ in each process.

We regard $v_1(\cdot),v_2(\cdot)$ as two strategies of player 1 and  2. For both of them, they cannot observe the process $x(\cdot),y(\cdot),z_1(\cdot),z_2(\cdot)$ directly. However, they can observe their own related processes $Y_1(\cdot)$, $Y_2(\cdot)$, which satisfy the following equations \footnote{Here we assume that the control variables $v_1(\cdot),v_2(\cdot)$ explicitly appeared in the observation function $h_i(\cdot)$, which is common in control problems under partial observation (see, e.g.,\cite{WangWuXiong13}).}

\begin{equation}\label{Observe}
\left\{
\begin{aligned}
dY_i(t)=&h_i(t,x(t),v_{1}(t),v_{2}(t))dt+dW_i^{v_1,v_2}(t),\\
Y_i(0)=&0\quad(i=1,2),\\
\end{aligned}
\right.
\end{equation}
where $h_i: \Omega\times[0,T]\times\mathbb{R}\times U_1\times U_2\mapsto\mathbb{R},i=1,2$ is a continuous map. We define the filtration $\mathcal{F}_t^i=\sigma\{Y_i(s)|0\leq s\leq t\},i=1,2$ as the information for player $i$ obtained at time $t$, and the admissible control $v_i(\cdot)$ as
\begin{equation}
v_i(t)\in\mathcal{U}_i=\{v_i(\cdot)\in U_i|v_i(t)\in\mathcal{F}_t^i\ \ \text{and} \sup\limits_{0\leq t\leq T}\mathbb{E}|v_i(t)|^8<\infty,a.e\}\quad(i=1,2),
\end{equation}
where $\mathcal{U}_i,i=1,2$ is called the open-loop admissible control set for player $i$.

\textbf{Hypothesis(H1).} Suppose functions $b,\sigma,\sigma_1,\sigma_2,h_1,h_2,f,g$ are continuously differentiable in $(x,y,z,z_1,z_2,v_1,v_2)$. The partial derivatives $b_x,b_{v_i},\sigma_x,\sigma_{v_i},\sigma_{jx}$,$\sigma_{jv_i},h_{jx},h_{jv_i}$,$f_x,\\
f_y,f_z,f_{z_j},f_{v_i},g_x,i,j=1,2$ are uniformly bounded. Further, we assume there is constant $C$ such that $|h(t,x,v_1,v_2)|$
$+|\sigma_1(t,x,v_1,v_2)|+|\sigma_2(t,x,v_1,v_2)|\leq C$ for $\forall (t,x,v_1,v_2)\in[0,T]\times\mathbb{R}\times U_1\times U_2$.

From the Hypothesis(H1), we can defined a new probability measure $\mathbb{P}^{v_1,v_2}$ by

\begin{equation}
\frac{d\mathbb{P}^{v_1,v_2}}{d\mathbb{P}}\bigg|_{\mathcal{F}_t}=Z^{v_1,v_2}(t),
\end{equation}
where $Z^{v_1,v_2}(\cdot)$ is a $\mathcal{F}_t$-martingale
\begin{equation}
Z^{v_1,v_2}(t)=\exp\big\{\sum_{j=1}^2\int_0^th_j(s,x(s),v_1(s),v_2(s))dY_j(s)-\frac{1}{2}\sum_{j=1}^2\int_0^th^2_j(s,x(s),v_1(s),v_2(s))ds\big\}.
\end{equation}
Equivalently, it can be written in the SDE form
\begin{equation}
\left\{
\begin{aligned}
dZ^{v_1,v_2}(t)=&h_1(t,x(t),v_{1}(t),v_{2}(t))Z^{v_1,v_2}(t)dY_1(t)+h_2(t,x(t),v_{1}(t),v_{2}(t))Z^{v_1,v_2}(t)dY_2(t),\\
Z^{v_1,v_2}(0)=&1.\\
\end{aligned}
\right.
\end{equation}

By using the Girsanov theorem,  $(W(\cdot),W_1^{v_1,v_2}(\cdot),W_2^{v_1,v_2}(\cdot))$ becomes a 3-dimensional standard Brownian motion defined on $(\Omega,\mathcal{F},\{\mathcal{F}_{t}\}_{t\geq0},\mathbb{P}^{v_1,v_2})$, where ($W_1^{v_1,v_2}(\cdot),W_2^{v_1,v_2}(\cdot))$ is a 2-dimensional controlled Brownian motion and $Y_j(\cdot),j=1,2$  turn out to be a stochastic observation process.

Based on the construction above, we define two cost functional under  space $(\Omega,\mathcal{F},\{\mathcal{F}_{t}\}_{t\geq0},\mathbb{P}^{v_1,v_2})$.
\begin{equation}\label{cost functional}
\begin{aligned}
J_i(v_1(\cdot),v_2(\cdot))=\mathbb{E}^{v_1,v_2}[\int_0^Tl_i(t,x(t),y(t),z(t),z_1(t),z_2(t),v_{1}(t),v_{2}(t))dt+\Phi_i(x(T))+\gamma_i(y(0))]\\
=\mathbb{E}[\int_0^TZ^{v_1,v_2}(t)l_i(t,x(t),y(t),z(t),z_1(t),z_2(t),v_{1}(t),v_{2}(t))dt+Z^{v_1,v_2}(T)\Phi_i(x(T))+\gamma_i(y(0))],
\end{aligned}
\end{equation}
for two players $i=1,2$, where $\mathbb{E}^{v_1,v_2}$ is the corresponding expectation. $l_i: \Omega\times[0,T]\times\mathbb{R}\times\mathbb{R}\times\mathbb{R}\times\mathbb{R}\times\mathbb{R}\times U_1\times U_2\mapsto\mathbb{R}$, $\Phi_i:\Omega\times\mathbb{R}\mapsto\mathbb{R}$, $\gamma_i:\mathbb{R}\mapsto\mathbb{R}$ are continuous maps. In this cost functional, it contains the running cost part representing an utility in duration, and the terminal and initial representing the restrict on the endpoints.

\textbf{Hypothesis(H2).}  Suppose functions $l_i$, $\Phi_i$, $\gamma_i,i=1,2$ are continuously differentiable in $(x,y,z,z_1,z_2$,
$v_1,v_2)$, $x$, $y$ respectively, the partial derivatives $l_{ix},l_{iy},l_{iz},l_{iz_j},l_{iv_j},i,j=1,2$ are bounded by $C(1+|y|+|z|+|z_1|+|z_2|+|v_1|+|v_2|)$ where $C$ is a constant.

For each of the player, his goal is to minimise his own cost. Here we set $(u_1,u_2)\in\mathcal{U}_1\times\mathcal{U}_2$ such that
\begin{equation}\label{nzJ}
\left\{
\begin{aligned}
J_1(u_1(\cdot),u_2(\cdot))=\min\limits_{v_1(\cdot)\in \mathcal{U}_1}J_1(v_1(\cdot),u_2(\cdot)),\\
J_2(u_1(\cdot),u_2(\cdot))=\min\limits_{v_2(\cdot)\in \mathcal{U}_2}J_2(u_1(\cdot),v_2(\cdot)).\\
\end{aligned}
\right.
\end{equation}

In this definition, $(u_1,u_2)$ is the well-known open-loop Nash equilibrium point of our partially-observed forward-backward non-zero sum system, and $(x,y,z,z_1,z_2,Z)$ is the corresponding equilibrium state process. Similar to the optimal control, what we want to do is to find this equilibrium control. We denote the whole problem above as \textbf{Problem(NEP)}.

In particular, If we set $J(v_1(\cdot),v_2(\cdot))=J_1(v_1(\cdot),v_2(\cdot))=-J_2(v_1(\cdot),v_2(\cdot))$, then (\ref{nzJ}) is equivalent to
\begin{equation}
J(u_1(\cdot),v_2(\cdot))\leq J(u_1(\cdot),u_2(\cdot))\leq J(v_1(\cdot),u_2(\cdot)),
\end{equation}
for $\forall (v_1(\cdot),v_2(\cdot))\in \mathcal{U}_1\times\mathcal{U}_2$.

In that case, the reward of player 1 is actually the cost of player 2, and the sum is always zero. We can regard it as a special case of non-zero sum game. We define this problem of our system as \textbf{Problem(EP)}.
\begin{remark}\label{circle}
If we at first suppose $W_j(\cdot)=W_j^{v_1,v_2}(\cdot),j=1,2$ to be a $\mathcal{F}_t$-Brownian motion under $\mathbb{P}$, then the distribution of observation process $Y(\cdot)$ will be depending on the control process. In that way, our admission control is adapted to a controlled filtration, which appears a circulation. Here, we break through the circulation by Girsanov theorem, making observation process to be an uncontrolled stochastic process and depict the controlled Brownian motion under related equivalent probability measure.
\end{remark}

\section{Maximum principle}
In this section, we will establish the necessary condition (maximum principle) of existence of open-loop Nash equilibrium point in problem (NEP), and give a sufficient condition (verification theorem) of a special class of system.

\subsection{Variational equation}
Let $(v_1(\cdot),v_2(\cdot))\in\mathbb{L}^8_{\mathcal{F}^1}(0,T;\mathbb{R})\times\mathbb{L}^8_{\mathcal{F}^2}(0,T;\mathbb{R})$ such that $(u_1(\cdot)+v_1(\cdot),u_2(\cdot)+v_2(\cdot))\in\mathcal{U}_1\times\mathcal{U}_2$.

For any $\epsilon\in[0,1]$, we make the variational controls as

\begin{equation}
\begin{aligned}
u_1^{\epsilon}(\cdot)=u_1(\cdot)+\epsilon v_1(\cdot),\\
u_2^{\epsilon}(\cdot)=u_2(\cdot)+\epsilon v_2(\cdot).
\end{aligned}
\end{equation}

Because $\mathcal{U}_1,\mathcal{U}_2$ are convex sets, we have $(u_1^{\epsilon}(\cdot),u_2^{\epsilon}(\cdot))\in\mathcal{U}_1\times\mathcal{U}_2$. We denote
\begin{equation*}
\phi^{u_i^\epsilon}(\cdot),\quad\phi=x,y,z,z_1,z_2,Z\quad (i=1,2),
\end{equation*}

as the corresponding state processes of variation $(u_1^\epsilon,u_2)$ or $(u_1,u_2^\epsilon)$.

It is noteworthy that when using the variational technique, we had better require the Brownian motion do not affected by the control process. Then our state equation can be written as

\begin{equation}\label{NewState}
\left\{
\begin{aligned}
dx(t)=&\Big[b(t,x(t),v_{1}(t),v_{2}(t))-\sum_{j=1}^2\sigma_j(t,x(t),v_{1}(t),v_{2}(t))h_j(t,x(t),v_{1}(t),v_{2}(t))\Big]dt\\
+&\sigma(t,x(t),v_{1}(t),v_{2}(t))dW(t)+\sum_{j=1}^2\sigma_{j}(t,x(t),v_{1}(t),v_{2}(t))dY_{j}(t),\\
-dy(t)=&f(t,x(t),y(t),z(t),z_1(t),z_2(t),v_{1}(t),v_{2}(t))dt-z(t)dW(t)-\sum_{j=1}^2z_j(t)dY_j(t),\\
x(0)=&x_0,\\
y(t)=&g(x(T)),
\end{aligned}
\right.
\end{equation}
where $(W(\cdot),Y_1(\cdot),Y_2(\cdot))$ is $\mathcal{F}_t$-Brownian motion under $\mathbb{P}$.

Then we have the following estimates under Hypothesis (H1).
\begin{lem}
\begin{equation}
\sup\limits_{0\leq t\leq T}\mathbb{E}|x(t)|^8\leq C(1+\sup\limits_{0\leq t\leq T}\mathbb{E}|v(t)|^8),
\end{equation}

\begin{equation}
\sup\limits_{0\leq t\leq T}\mathbb{E}|y(t)|^2\leq C(1+\sup\limits_{0\leq t\leq T}\mathbb{E}|v(t)|^2),
\end{equation}

\begin{equation}
\mathbb{E}\big(\int_0^T|z(t)|^2dt+\int_0^T|z_1(t)|^2dt+\int_0^T|z_2(t)|^2dt\big)\leq C(1+\sup\limits_{0\leq t\leq T}\mathbb{E}|v(t)|^2),
\end{equation}

\begin{equation}
\sup\limits_{0\leq t\leq T}\mathbb{E}|Z^{v_1,v_2}(t)|\leq K,
\end{equation}
where $C,K$ is constant independent of $\epsilon$.
\end{lem}

\begin{lem}
\begin{equation}
\sup\limits_{0\leq t\leq T}\mathbb{E}|x^{u_i^\epsilon}(t)-x^{}(t)|^8\leq C\epsilon^8,
\end{equation}

\begin{equation}
\sup\limits_{0\leq t\leq T}\mathbb{E}|y^{u_i^\epsilon}(t)-y(t)|^2\leq C\epsilon^2,
\end{equation}

\begin{equation}
\mathbb{E}\int_0^T|z^{u_i^\epsilon}(t)-z(t)|^2dt\leq C\epsilon^2,
\end{equation}

\begin{equation}
\mathbb{E}\int_0^T|z_1^{u_i^\epsilon}(t)-z_1(t)|^2dt\leq C\epsilon^2,
\end{equation}

\begin{equation}
\mathbb{E}\int_0^T|z_2^{u_i^\epsilon}(t)-z_2(t)|^2dt\leq C\epsilon^2,
\end{equation}

\begin{equation}
\sup\limits_{0\leq t\leq T}\mathbb{E}|Z^{u_i^\epsilon}(t)-Z(t)|^2\leq C\epsilon^2,
\end{equation}
for $i=1,2.$, where $C$ is constant independent of $\epsilon$.
\end{lem}

For notation simplicity, we set
\[\zeta(t)=\zeta(t,x(t),u_{1}(t),u_{2}(t)),\  \text{for}\ \ \zeta=b, \sigma, \sigma_i,h_i\quad (i=1,2),\]
\[\psi(t)=\psi(t,x(t),y(t),z(t),z_1(t),z_2(t),u_{1}(t),u_{2}(t)),\ \text{for}\  \psi=f,l_i\quad (i=1,2).\]

We introduce the following variational equations
\begin{equation}\label{Variational xy}
\left\{
\begin{aligned}
dx_i^1(t)=&\big\{[b_x(t)-\sum_{j=1}^2(\sigma_{jx}(t)h_j(t)+\sigma_j(t)h_{jx}(t))]x_i^1(t)+[b_{v_i}(t)-\sum_{j=1}^2(\sigma_{jv_{i}}(t)h_j(t)+\sigma_j(t)h_{jv_{i}}(t))]v_i(t)\big\}dt\\
+&[\sigma_{x}(t)x_i^1(t)+\sigma_{v_i}(t)v_i(t)]dW(t)+\sum_{j=1}^2[\sigma_{jx}(t)x^1_i(t)+\sigma_{jv_i}(t)v_i(t)]dY_j(t),\\
-dy_i^1(t)=&[f_x(t)x_i^1(t)+f_y(t)y_i^1(t)+f_z(t)z_i^1(t)+\sum_{j=1}^2f_{z_j}(t)z_{ji}^1(t)+f_{v_i}(t)v_i(t)]dt-z_i^1(t)dW(t)-\sum_{j=1}^2z_{ji}^1(t)dY_j(t),\\
x_i^1(0)=&0,\\
y_i^1(T)=&g_x(x(T))x_i^1(T)\quad(i=1,2),
\end{aligned}
\right.
\end{equation}
and

\begin{equation}\label{Variational z}
\left\{
\begin{aligned}
dZ_i^1(t)=&\sum_{j=1}^2\big[Z_i^1(t)h_j(t)+Z(t)(h_{jx}(t)x_i^1(t)+h_{jv_i}(t)v_i(t))\big]dY_j(t),\\
Z_i^1(0)=&0\quad(i=1,2).\\
\end{aligned}
\right.
\end{equation}

From Hypothesis (H1), we know that (\ref{Variational xy}) and (\ref{Variational z}) exist a unique solution respectively.

Next, we make the notation
\[
\phi_i^{\epsilon}(t)=\frac{\phi^{u_i^{\epsilon}}(t)-\phi(t)}{\epsilon}-\phi_i^1(t),\ \ \text{for}\ \ \phi=x,y,z,z_1,z_2,Z\quad(i=1,2),
\]
and
\[
\bar{\phi}(t)=\phi^{u_i^{\epsilon}}(t)-\phi(t),\ \ \text{for}\ \ \phi=x,y,z,z_1,z_2,Z,b,\sigma,\sigma_1,\sigma_2,h_1,h_2,Z\quad(i=1,2).
\]
Then we have
\begin{lem}\label{convergence}
For $i,j=1,2$,
\begin{equation}
\lim\limits_{\epsilon\rightarrow0}\sup\limits_{0\leq t\leq T}\mathbb{E}|x_i^\epsilon(t)|^4=0,
\end{equation}

\begin{equation}
\lim\limits_{\epsilon\rightarrow0}\sup\limits_{0\leq t\leq T}\mathbb{E}|Z_i^\epsilon(t)|^2=0,
\end{equation}

\begin{equation}
\lim\limits_{\epsilon\rightarrow0}\sup\limits_{0\leq t\leq T}\mathbb{E}|y_i^\epsilon(t)|^2=0,
\end{equation}

\begin{equation}
\lim\limits_{\epsilon\rightarrow0}\mathbb{E}\int_0^T|z_i^\epsilon(t)|^2dt=0,
\end{equation}

\begin{equation}
\lim\limits_{\epsilon\rightarrow0}\mathbb{E}\int_0^T|z_{ji}^\epsilon(t)|^2dt=0.
\end{equation}
\begin{proof}
We only consider the first two case for $i=1$. The rest are similar and are well-known results.

(i).
\begin{equation}
\begin{aligned}
dx_1^\epsilon(t)=&\bigg\{\big(\frac{\bar{b}(t)}{\epsilon}-b_x(t)x^1_1(t)-b_{v_i}(t)v_i(t)\big)-\sum_{j=1}^2\bigg[\big(\frac{\bar{\sigma}_j(t)}{\epsilon}-\sigma_{jx}(t)x_1^1(t)-\sigma_{jv_1}(t)v_1(t)\big)h_j(t)\\
+&\big(\frac{\bar{h}_j(t)}{\epsilon}-h_{jx}(t)x_1^1(t)-h_{jv_1}(t)v_1(t)\big)\sigma_j(t)+\frac{\bar{\sigma}_j(t)}{\epsilon}\bar{h}_j(t)\bigg]\bigg\}dt\\
+&\big(\frac{\bar{\sigma}(t)}{\epsilon}-\sigma_{x}(t)x_1^1(t)-\sigma_{v_1}(t)v_1(t)\big)dW(t)+\sum_{j=1}^2\big(\frac{\bar{\sigma}_j(t)}{\epsilon}-\sigma_{jx}(t)x_1^1(t)-\sigma_{jv_1}(t)v_1(t)\big)dY_j(t)\\
=&\bigg[x_1^\epsilon(t)b_x(t)+\frac{\bar{x}(t)}{\epsilon}\big(\int_0^1b_x(\Theta)d\lambda-b_x(t)\big)+v_1(t)\big(\int_0^1b_{v_1}(\Theta)d\lambda-b_{v_1}(t)\big)\bigg]dt\\
+&\bigg[x_1^\epsilon(t)\sigma_x(t)+\frac{\bar{x}(t)}{\epsilon}\big(\int_0^1\sigma_x(\Theta)d\lambda-\sigma_x(t)\big)+v_1(t)\big(\int_0^1\sigma_{v_1}(\Theta)d\lambda-\sigma_{v_1}(t)\big)\bigg]dW(t)\\
+&\sum_{j=1}^2\bigg\{\bigg[x_1^\epsilon(t)\sigma_{jx}(t)+\frac{\bar{x}(t)}{\epsilon}\big(\int_0^1\sigma_{jx}(\Theta)d\lambda-\sigma_{jx}(t)\big)+v_1(t)\big(\int_0^1\sigma_{jv_1}(\Theta)d\lambda-\sigma_{jv_1}(t)\big)\bigg]h_j(t)dt\\
+&\bigg[x_1^\epsilon(t)h_{jx}(t)+\frac{\bar{x}(t)}{\epsilon}\big(\int_0^1h_{jx}(\Theta)d\lambda-h_{jx}(t)\big)+v_1(t)\big(\int_0^1h_{jv_1}(\Theta)d\lambda-h_{jv_1}(t)\big)\bigg]\sigma_j(t)dt\\
+&\bigg[\frac{\bar{x}(t)}{\epsilon}\int_0^1\sigma_{jx}(\Theta)d\lambda+v_1(t)\int_0^1\sigma_{jv_1}(\Theta)d\lambda\bigg]\bar{h}_j(t)\bigg\}dt,
\end{aligned}
\end{equation}
where \[(\Theta)=(t,x(t)+\lambda\bar{x}(t),u_1(t)+\lambda\epsilon v(t),u_2(t)).\]

Thus we have
\begin{equation}
\begin{aligned}
\mathbb{E}|x_1^\epsilon(t)|^4\leq& C\mathbb{E}\int_0^t|x_1^\epsilon(s)|^4ds+C\sqrt{\mathbb{E}\int_0^t|\frac{\bar{x}(s)}{\epsilon}|^8ds}\bigg[\sqrt{\mathbb{E}\int_0^t\big(\int_0^1b_x(\Theta)d\lambda-b_x(s)\big)^8ds}\\
+&\sqrt{\mathbb{E}\int_0^t\big(\int_0^1\sigma_x(\Theta)d\lambda-\sigma_x(s)\big)^8ds}+\sum_{j=1}^2\sqrt{\mathbb{E}\int_0^t\big(\int_0^1\sigma_{jx}(\Theta)d\lambda-\sigma_{jx}(s)\big)^8ds}\\
+&\sum_{j=1}^2\sqrt{\mathbb{E}\int_0^t\big(\int_0^1h_{jx}(\Theta)d\lambda-h_{jx}(s)\big)^8ds}\bigg]\\
+&C\sqrt{\mathbb{E}\int_0^t|v_1(s)|^8ds}\bigg[\sqrt{\mathbb{E}\int_0^t\big(\int_0^1b_{v_1}(\Theta)d\lambda-b_{v_1}(s)\big)^8ds}\\
+&\sqrt{\mathbb{E}\int_0^t\big(\int_0^1\sigma_{v_1}(\Theta)d\lambda-\sigma_{v_1}(s)\big)^8ds}+\sum_{j=1}^2\sqrt{\mathbb{E}\int_0^t\big(\int_0^1\sigma_{jv_1}(\Theta)d\lambda-\sigma_{jv_1}(s)\big)^8ds}\\
+&\sum_{j=1}^2\sqrt{\mathbb{E}\int_0^t\big(\int_0^1h_{jv_1}(\Theta)d\lambda-h_{jv_1}(s)\big)^8ds}\bigg]\\
+&C\big(\sqrt{\mathbb{E}\int_0^t|\frac{\bar{x}(s)}{\epsilon}|^8ds}+\sqrt{\mathbb{E}\int_0^t|v_1(s)|^8ds}\big)\sqrt{\mathbb{E}\int_0^t\bar{x}(s)ds},
\end{aligned}
\end{equation}
where $C$ is a constant.
From Hypothesis (H1), we know that the right side of the inequality converges to 0 when $\epsilon\mapsto 0$. The case of $i=2$ is similar.

(ii).
\begin{equation}
\begin{aligned}
dZ_1^\epsilon(t)=&\sum_{j=1}^2\bigg[\frac{h_j^{u_1^\epsilon}(t)Z^{u_1^\epsilon}(t)-h_j(t)Z(t)}{\epsilon}-Z_1^1(t)h_j(t)-Z(t)h_{jx}(t)x_1^1(t)-Z(t)h_{jv_1}(t)v_1(t)\bigg]dY_j(t)\\
=&\sum_{j=1}^2\bigg[\frac{Z^{u_1^\epsilon}(t)-Z(t)}{\epsilon}h_j(t)+Z^{u_1^\epsilon}(t)\frac{h_j^{u_1^\epsilon}(t)-h_j(t)}{\epsilon}-Z_1^1(t)h_j(t)-Z(t)h_{jx}(t)x_1^1(t)\\
&\ \ \ \ \ \ -Z(t)h_{jv_1}(t)v_1(t)\bigg]dY_j(t)\\
=&\sum_{j=1}^2\bigg[Z_1^\epsilon(t)h_j(t)+Z(t)\frac{\bar{h}_j(t)}{\epsilon}-Z(t)h_{jx}(t)x_1^1(t)-Z(t)h_{jv_1}(t)v_1(t)+\frac{\bar{Z}(t)}{\epsilon}\bar{h}_j(t)\bigg]dY_j(t)\\
=&\sum_{j=1}^2\bigg[Z_1^\epsilon(t)h_j(t)+Z(t)\bigg(x_1^\epsilon(t)h_{jx}(t)+\frac{\bar{x}(t)}{\epsilon}\big(\int_0^1h_{jx}(\Theta)d\lambda-h_{jx}(t)\big)\\
&\ \ \ \ \ \ +v_1(t)\big(\int_0^1h_{jv_1}(\Theta)d\lambda-h_{jv_1}(t)\big)\bigg)+\big(Z_1^\epsilon(t)+Z_1^1(t)\big)\bar{h}_j(t)\bigg]dY_j(t)\\
=&\sum_{j=1}^2\bigg[Z_1^\epsilon(t)h_j^{u_1^\epsilon}(t)+Z(t)\bigg(x_1^\epsilon(t)h_{jx}(t)+\frac{\bar{x}(t)}{\epsilon}\big(\int_0^1h_{jx}(\Theta)d\lambda-h_{jx}(t)\big)\\
&\ \ \ \ \ \ +v_1(t)\big(\int_0^1h_{jv_1}(\Theta)d\lambda-h_{jv_1}(t)\big)\bigg)+Z_1^1(t)\bar{h}_j(t)\bigg]dY_j(t).\\
\end{aligned}
\end{equation}

Thus we have
\begin{equation}
\begin{aligned}
\mathbb{E}|Z_1^\epsilon(t)|^2\leq& C\bigg[\mathbb{E}\int_0^t|Z_1^\epsilon(s)|^2ds+\mathbb{E}\int_0^t|Z(s)x_1^\epsilon(s)|^2ds\\
+&\mathbb{E}\int_0^t|Z(s)\frac{\bar{x}(s)}{\epsilon}|^2\Big(\int_0^1h_{jx}(\Theta)d\lambda-h_{jx}(s)\Big)^2ds\\
+&\mathbb{E}\int_0^t|v_1(s)|^2\Big(\int_0^1h_{jv_1}(\Theta)d\lambda-h_{jv_1}(s)\Big)^2ds+\mathbb{E}\int_0^t|Z_1^1(s)\bar{h}_j(s)|^2ds\\
\leq&C\bigg[\mathbb{E}\int_0^t|Z_1^\epsilon(s)|^2ds+\sqrt{\mathbb{E}\int_0^t|x_1^\epsilon(s)|^4ds}\\
+&\sqrt[4]{\mathbb{E}\int_0^t|\frac{\bar{x}(s)}{\epsilon}|^8ds}\sqrt{\mathbb{E}\int_0^t\big(\int_0^1h_{jx}(\Theta)d\lambda-h_{jx}(s)\big)^4ds}\\
+&\sqrt{\mathbb{E}\int_0^t|v_1(s)|^4ds}\sqrt{\mathbb{E}\int_0^t\big(\int_0^1h_{jv_1}(\Theta)d\lambda-h_{jv_1}(s)\big)^4ds}\\
+&\sqrt{\mathbb{E}\int_0^t|Z_1^1(s)|^4ds}\sqrt{\mathbb{E}\int_0^t|\bar{x}(s)|^4ds},
\end{aligned}
\end{equation}
where $C$ is a constant.
From Hypothesis (H1), we know that the right side of the inequality converges to 0 when $\epsilon\mapsto 0$. The case of $i=2$ is similar.
\end{proof}
\end{lem}

\subsection{Variational inequality}
From the definition of open-loop Nash equilibrium point $(u_1(\cdot),u_2(\cdot))$ in Problem (NEP), it is clear that

\begin{equation}\label{nec}
\begin{aligned}
\epsilon^{-1}[J_1(u_1^\epsilon(\cdot),u_2(\cdot))-J_1(u_1(\cdot),u_2(\cdot))]\geq 0,\\
\epsilon^{-1}[J_2(u_1(\cdot),u_2^\epsilon(\cdot))-J_2(u_1(\cdot),u_2(\cdot))]\geq 0.
\end{aligned}
\end{equation}

Let $\Gamma_i(\cdot)=Z_i^1(\cdot)Z_i^{-1}(\cdot),i=1,2.$, From It\^{o}'s formula, we have

\begin{equation}\label{variational Gamma}
\left\{
\begin{aligned}
d\Gamma_i(t)=&\sum_{j=1}^{2}\big[h_{jx}(t)x_i^1(t)+h_{jv_i}(t)v_i(t)\big]dW^{u_1,u_2}_j(t),\\
\Gamma_i(0)=&0\quad(i=1,2).
\end{aligned}
\right.
\end{equation}

From (\ref{nec}), we can derive the variational inequality.
\begin{lem}
\begin{equation}
\begin{aligned}
\mathbb{E}&^{u_1,u_2}\bigg[\Phi_{ix}(x(T))x_i^1(T)+\gamma_{iy}(y(0))y_i^1(0)+\Phi_i(x(T))\Gamma_i(T)+\int_0^T\Gamma_i(t)l_i(t)dt\\
+&\int_0^T[l_{ix}(t)x_i^1(t)+l_{iy}(t)y_i^1(t)+l_{iz}(t)z_i^1(t)+\sum_{j=1}^{2}l_{iz_j}(t)z_{ji}^1(t)]dt+\int_0^Tl_{iv_i}(t)v_i(t)dt\bigg]\geq 0.
\end{aligned}
\end{equation}
\end{lem}

\begin{proof}
we only consider the case $i=1$.
From (\ref{cost functional}), we have
\begin{equation}
\begin{aligned}
\epsilon^{-1}[J_1(u_1^\epsilon(\cdot),u_2(\cdot))-&J_1(u_1(\cdot),u_2(\cdot))]=\epsilon^{-1}\mathbb{E}\bigg[\int_0^T\Big(Z^{u_1^\epsilon}(t)l_1^{u_1^\epsilon}(t)-Z(t)l_1(t)\Big)dt\\
+&\Big(Z^{u_1^\epsilon}(T)\Phi_1^{u_1^\epsilon}(x(T))-Z(T)\Phi_1(x(T))\Big)+\Big(\gamma_1^{u_1^\epsilon}(y(0))-\gamma_1(y(0))\Big)\bigg]\geq 0.
\end{aligned}
\end{equation}

According to Lemma \ref{convergence} and Hypothesis (H2),
\begin{equation}\label{1}
\begin{aligned}
\epsilon^{-1}[\gamma_1^{u_1^\epsilon}(y(0))-\gamma_1(y(0))]=\int_0^1\gamma_{1y}\Big(y(0)+\lambda\big(y^{u_1^\epsilon}(0)-y(0)\big)\Big)d\lambda\frac{(y^{u_1^\epsilon}(0)-y(0))}{\epsilon}\rightarrow\gamma_{iy}(y(0))y_1^1(0),
\end{aligned}
\end{equation}
and
\begin{equation}\label{2}
\begin{aligned}
\epsilon^{-1}\mathbb{E}&[Z^{u_1^\epsilon}(T)\Phi_1^{u_1^\epsilon}(x(T))-Z(T)\Phi_1(x(T))]\\
=&\epsilon^{-1}\mathbb{E}[Z^{u_1^\epsilon}(t)(\Phi_1^{u_1^\epsilon}(x(T))-\Phi_1(x(T)))+\Phi_1(x(T))(Z^{u_1^\epsilon}(T)-Z(T))]\\
=&\mathbb{E}[Z^{u_1^\epsilon}(T)\Bigg(\int_0^1\Phi_{1x}\Big(x(T)+\lambda\big(x^{u_1^\epsilon}(T)-x(T)\big)\Big)d\lambda\frac{(x^{u_1^\epsilon}(T)-x(T))}{\epsilon}\Bigg)+\Phi_1(x(T))\frac{(Z^{u_1^\epsilon}(T)-Z(T))}{\epsilon}]\\
\rightarrow&\mathbb{E}[Z(T)\Phi_{1x}(x(T))x_1^1(T)+\Phi_1(x(T))Z_1^1(T)].
\end{aligned}
\end{equation}
Similarly, we have
\begin{equation}\label{3}
\begin{aligned}
\epsilon^{-1}\mathbb{E}[\int_0^T(Z^{u_1^\epsilon}(t)l_1^{u_1^\epsilon}(t)-Z(t)l_1(t))dt]\rightarrow&\mathbb{E}\bigg[\int_0^TZ(t)\big(l_{1x}(t)x_1^1(t)+l_{1y}(t)y_1^1(t)+l_{1z}(t)z_1^1(t)+\sum_{j=1}^{2}l_{1z_j}(t)z_{j1}^1(t)\\
+&l_{1v_1}(t)v_1(t)\big)dt+\int_0^Tl_1(t)Z_1^1(t)dt\bigg].
\end{aligned}
\end{equation}
From the definition of $\Gamma(\cdot)$ and (\ref{1})-(\ref{3}), we derive the variational inequality.
\end{proof}

\subsection{A necessary condition (maximum principle)}
In the following, we ignore the superscript of $W_j^{u_1,u_2}(\cdot),j=1,2$ for notation simplicity.
We formulate adjoint equaions under probability measure $\mathbb{P}^{u_1,u_2}$.

\begin{equation}\label{P}
\left\{
\begin{aligned}
-dP_{i}(t)=&l_i(t)dt-Q_{i}(t)dW(t)-\sum_{j=1}^2Q_{ji}dW_j(t),\\
P_i(T)=&\Phi_i(x(T))\quad(i=1,2).
\end{aligned}
\right.
\end{equation}

\begin{equation}\label{pq}
\left\{
\begin{aligned}
dp_i(t)=&[f_y(t)p_i(t)-l_{iy}(t)]dt+[f_z(t)p_i(t)-l_{iz}(t)]dW(t)+\sum_{j=1}^2\big[\big(f_{z_j}(t)-h_j(t)\big)p_i(t)-l_{iz_j}(t)\big]dW_j(t),\\
-dq_i(t)=&\Big\{\big[b_x(t)-\sum_{j=1}^2\sigma_{j}(t)h_{jx}(t)\big]q_i(t)+\sigma_x(t)k_i(t)+\sum_{j=1}^2\big[\sigma_{jx}(t)k_{ji}(t)+h_{jx}(t)Q_{ji}(t)\big]\\
-&f_x(t)p_i(t)+l_{ix}(t)\Big\}dt-k_i(t)dW(t)-\sum_{j=1}^2k_{ji}(t)dW_j(t),\\
p_i(0)=&-\gamma_y(y(0)),\\
q_i(T)=&-g_x(x(T))p_i(T)+\Phi_{ix}(x(T))\quad(i=1,2).
\end{aligned}
\right.
\end{equation}

It is noteworthy here that  $W_j(\cdot),j=1,2$ is the Brownian motion under probability measure $\mathbb{P}^{u_1,u_2}$. Due to the observation equation (\ref{Observe}) and the appearance of $W_j^{v_1,v_2}(\cdot)$ and $Y_j(\cdot),j=1,2$ in the forward-backward state equation  (\ref{FBSDE}), the equation (\ref{pq}) is not the classic form any more. Also we should introduce equation (\ref{P}) to deal with the controlled probability measure $\mathbb{P}^{v_1,v_2}$ or expectation $\mathbb{E}^{v_1,v_2}$ related to observation process when using the variational method.

Now we give the necessary condition.
\begin{theorem}\label{MP}
Suppose (H1) and (H2) hold, $(u_1(\cdot),u_2(\cdot))$ is an open-loop Nash equilibrium point of problem (NEP), and $(x,y,z,z_1,z_2)$ is the corresponding state process, then we have
\begin{equation}
\begin{aligned}
\mathbb{E}^{u_1,u_2}[\tilde{H}_{1{v_1}}(t,x,y,z,z_1,z_2,u_1,u_2;q_1,k_1,k_{11},k_{21},p_1,Q_{11},Q_{21})(v_1-u_1(t))|\mathcal{F}_t^1]\geq 0,\\
\mathbb{E}^{u_1,u_2}[\tilde{H}_{2{v_2}}(t,x,y,z,z_1,z_2,u_1,u_2;q_2,k_2,k_{12},k_{22},p_2,Q_{12},Q_{22})(v_2-u_2(t))|\mathcal{F}_t^2]\geq 0,
\end{aligned}
\end{equation}
for $\forall (v_1,v_2)\in U_1\times U_2,\ a.e.$. Here we set

\begin{equation}\label{H}
\begin{aligned}
\tilde{H}_{iv_i}(t)=& \tilde{H}_{i{v_i}}(t,x,y,z,z_1,z_2,u_1,u_2;q_i,k_i,k_{1i},k_{1i},p_i,Q_{1i},Q_{2i})\\
=& H_{iv_i}(t)-\sum_{j=1}^2q_i(t)\sigma_j(t,x(t),u_1(t),u_2(t))h_{jv_i}(t,x(t),u_1(t),u_2(t)),
\end{aligned}
\end{equation}
where
\begin{equation}
\begin{aligned}
H_{i}(\cdot)\triangleq& b(t,x,u_1,u_2)q_i(t)+\sigma(t,x,u_1,u_2)k_i(t)+\sum_{j=1}^2[\sigma_j(t,x,u_1,u_2)k_{ji}(t)+h_j(t,x,u_1,u_2)Q_{ji}(t)]\\
-&[f(t,x,y,z,z_1,z_2,u_1,u_2)-\sum_{j=1}^2h_j(t,x,u_1,u_2)z_j(t)]p_i(t)+l_i(t,x,y,z,z_1,z_2,u_1,u_2)\ (i=1,2).\\
\end{aligned}
\end{equation}
\end{theorem}

\begin{proof}
We only consider the $i=1$ case. Applying It\^o's formula to $q_1(\cdot)x_1^1(\cdot)$, $p_1(\cdot)y_1^1(\cdot)$, $P_1(\cdot)\Gamma_1(\cdot)$ respectively, we have

\begin{equation}\label{qx}
\begin{aligned}
\mathbb{E}&^{u_1,u_2}[\big(\Phi_{1x}(x(T))-g_x(x(T))p_1(T)\big)x_1^1(T)]=\mathbb{E}^{u_1,u_2}\int_0^T\bigg[q(t)\Big([b_x(t)-\sum_{j=1}^2\sigma_j(t)h_{jx}(t)]x_1^1(t)\\
+&[b_{v_1}(t)-\sum_{j=1}^2\sigma_j(t)h_{jv_1}(t)]v_1(t)\Big)-x_1^1(t)\Big([b_x(t)-\sum_{j=1}^2\sigma_j(t)h_{jx}(t)]q_1(t)+\sigma_x(t)k_1(t)\\
+&\sum_{j=1}^2(\sigma_{jx}(t)k_{j1}(t)+h_{jx}(t)Q_{j1}(t)))-f_x(t)p_1(t)+l_{1x}(t)\Big)\bigg]dt+\mathbb{E}^{u_1,u_2}\bigg[\int_0^Tk_1(t)(\sigma_x(t)x_1^1(t)+\sigma_{v_1}(t)v_1(t))dt\bigg]\\
+&\sum_{j=1}^2\mathbb{E}^{u_1,u_2}\bigg[\int_0^Tk_{j1}(t)(\sigma_{jx}(t)x_1^1(t)+\sigma_{jv_1}(t)v_1(t))dt\bigg],
\end{aligned}
\end{equation}

\begin{equation}\label{py}
\begin{aligned}
\mathbb{E}&^{u_1,u_2}[(p_{1}(T)g_x(x(T))x_1^1(T)+\gamma_y(y(0))y_1^1(0)]=\mathbb{E}^{u_1,u_2}\int_0^T\bigg[-p_1(t)\bigg(f_x(t)x_1^1(t)+f_y(t)y_1^1(t)\\
+&f_z(t)z_1^1(t)+f_{v_1}(t)v_1(t)+\sum_{j=1}^2\big(f_{z_j}(t)-h_j(t)\big)z_{j1}^1(t)\bigg)+y_1^1(t)\big(f_y(t)p_1(t)-l_{iy}(t)\big)\bigg]dt\\
+&\mathbb{E}^{u_1,u_2}\bigg[\int_0^Tz_1^1(t)\big(f_z(t)p_1(t)-l_{1z}(t)\big)dt\bigg]+\sum_{j=1}^2\mathbb{E}^{u_1,u_2}\bigg[\int_0^Tz_{j1}^1(t)\Big((f_{z_j}(t)-h_{j}(t))p_1(t)-l_{1z_j}(t)\Big)dt\bigg],
\end{aligned}
\end{equation}
and
\begin{equation}\label{GammaP}
\begin{aligned}
\mathbb{E}&^{u_1,u_2}[\Phi_1(x(T))\Gamma_1(T)]=\mathbb{E}^{u_1,u_2}\big[\int_0^T-\Gamma_1(t)l_1(t)dt]+\sum_{j=1}^2\mathbb{E}^{u_1,u_2}\big[\int_0^TQ_{j1}(t)\big(h_{jx}(t)x_1^1(t)+h_{jv_1}(t)v_1(t)\big)\big]dt.
\end{aligned}
\end{equation}
Substituting (\ref{qx})-(\ref{GammaP}) into variational inequality, we get

\begin{equation}
\begin{aligned}
\mathbb{E}^{u_1,u_2}&\bigg[\int_0^T\Big(\big(b_{v_1}(t)-\sum_{j=1}^2\sigma_j(t)h_{jv_1}(t)\big)q_1(t)-f_{v_1}(t)p_1(t)+\sigma_{v_1}(t)k_1(t)+\sum_{j=1}^2\big(\sigma_{jv_1}(t)k_{j1}(t)+h_{jv_1}(t)Q_{j1}(t)\big)\\
+&l_{1v_1}(t)\Big)\cdot v_1(t)\bigg]dt\geq 0,
\end{aligned}
\end{equation}
for any $v_1(\cdot)$ such that $u_1(\cdot)+v_1(\cdot)\in\mathcal{U}_1$.

Let $w_1(\cdot)=u_1(\cdot)+v_1(\cdot)$, then above equation implies that

\begin{equation}
\mathbb{E}^{u_1,u_2}[\tilde{H}_{1v_1}(t)\cdot(w_1(t)-u_1(t))]\geq0, \qquad a.e.
\end{equation}

We set
\begin{equation}
\bar{w}_1(t)=v_1 1_A+u_1(t)1_{\Omega-A},\quad \forall v_1\in U_1,\quad\forall A\in\mathcal{F}_t^1.
\end{equation}

Then
\begin{equation}
v_1(t)=(v_1-u_1(t))1_A,\quad\forall v_1\in U_1,\quad\forall A\in\mathcal{F}_t^1.
\end{equation}

So
\begin{equation}
\mathbb{E}^{u_1,u_2}[1_A\tilde{H}_{1v_1}(t)\cdot(v_1-u_1(t))]\geq0,\quad\forall v_1\in U_1,\quad\forall A\in\mathcal{F}_t^1.
\end{equation}

Thus we get
\begin{equation}
\mathbb{E}^{u_1,u_2}[\tilde{H}_{1v_1}(t)\cdot(v_1-u_1(t))|\mathcal{F}_t^1]\geq0,\quad\forall v_1\in U_1.
\end{equation}

The prove of case $i=2$ is similar.
\end{proof}

\begin{remark}
The appearance of second part in the right-side of equation (\ref{H}) is caused by the appearance of control variables $v_1(\cdot),v_2(\cdot)$ of $h(\cdot) $ in observation equation (\ref{Observe}).
\end{remark}

\begin{cor}\label{MPsaddle}
Suppose (H1) and (H2) hold, $(u_1(\cdot),u_2(\cdot))$ is a saddle point of problem (EP), then we have

\begin{equation}
\begin{aligned}
\mathbb{E}^{u_1,u_2}[\tilde{H}_{1{v_1}}(t,x,y,z,z_1,z_2,u_1,u_2;q_1,k_1,k_{11},k_{21},p_1,Q_{11},Q_{21})(v_1-u_1(t))|\mathcal{F}_t^1]\geq 0,\\
\mathbb{E}^{u_1,u_2}[\tilde{H}_{1{v_2}}(t,x,y,z,z_1,z_2,u_1,u_2;q_1,k_1,k_{11},k_{21},p_1,Q_{11},Q_{21})(v_2-u_2(t))|\mathcal{F}_t^2]\leq 0,
\end{aligned}
\end{equation}
for $\forall (v_1,v_2)\in U_1\times U_2,\ a.e.$.
\end{cor}

\begin{remark}
Theorem \ref{MP} (corollary \ref{MPsaddle}) is equivalence to

\begin{equation}
\begin{aligned}
\mathbb{E}^{u_1,u_2}[\tilde{H}_{1{v_1}}(t,x,y,z,z_1,z_2,u_1,u_2;q_1,k_1,k_{11},k_{21},p_1,Q_{11},Q_{21})|\mathcal{F}_t^1]=0,\\
\mathbb{E}^{u_1,u_2}[\tilde{H}_{1{v_2}}(t,x,y,z,z_1,z_2,u_1,u_2;q_2,k_2,k_{12},k_{22},p_2,Q_{12},Q_{22})|\mathcal{F}_t^2]=0,\\
\end{aligned}
\end{equation}
for $\forall (v_1,v_2)\in U_1\times U_2,\ a.e.$.
\begin{equation}
\left(
\begin{aligned}
\mathbb{E}^{u_1,u_2}[\tilde{H}_{1{v_1}}(t,x,y,z,z_1,z_2,u_1,u_2;q_1,k_1,k_{11},k_{21},p_1,Q_{11},Q_{21})|\mathcal{F}_t^1]=0\\
\mathbb{E}^{u_1,u_2}[\tilde{H}_{1{v_2}}(t,x,y,z,z_1,z_2,u_1,u_2;q_1,k_1,k_{11},k_{21},p_1,Q_{11},Q_{21})|\mathcal{F}_t^2]=0
\end{aligned}
\right),
\end{equation}
for $\forall (v_1,v_2)\in U_1\times U_2,\ a.e.$.
\end{remark}

In the following remarks, we discuss some special cases in the system of our problem (NEP).
\begin{remark}
If the form of the forward equation $x(\cdot)$ in (\ref{FBSDE}) satisfies
\begin{equation}
\left\{
\begin{aligned}
dx(t)=&b(t,x(t),v_{1}(t),v_{2}(t))dt+\sigma(t,x(t),v_{1}(t),v_{2}(t))dW(t),\\
x(0)=&x,
\end{aligned}
\right.
\end{equation}
where it doesn't contain the $W_j^{u_1,u_2}(\cdot)$ part. Then, It is a special case that $\sigma_j(\cdot)\equiv k_{j\cdot}(\cdot)\equiv 0$. According to the theorem \ref{MP}, the necessary condition becomes:

\begin{equation}
\begin{aligned}
\mathbb{E}^{u_1,u_2}[\bar{H}_{i{v_i}}(t,x,y,z,z_1,z_2,u_1,u_2;q_i,k_i,k_{1i},k_{2i},p_i,Q_{1i},Q_{2i})|\mathcal{F}_t^i]=0\quad(i=1,2),
\end{aligned}
\end{equation}
where

\begin{equation}
\begin{aligned}
\bar{H}_{i}(\cdot)=& b(t,x,u_1,u_2)q_i(t)+\sigma(t,x,u_1,u_2)k_i(t)+\sum_{j=1}^2h_j(t,x,u_1,u_2)Q_{ji}(t)\\
-&f(t,x,y,z,z_1,z_2,u_1,u_2)p_i(t)+l_i(t,x,y,z,z_1,z_2,u_1,u_2).\\
\end{aligned}
\end{equation}

The adjoint process are changed to
\begin{equation}
\left\{
\begin{aligned}
dp_i(t)=&[f_y(t)p_i(t)-l_{iy}(t)]dt+[f_z(t)p_i(t)-l_{iz}(t)]dW(t)+\sum_{j=1}^2[(f_{z_j}(t)-h_j(t))p_i(t)-l_{iz_j}(t)]dW_j(t),\\
-dq_i(t)=&\big[b_x(t)q_i(t)+\sigma_x(t)k_i(t)+\sum_{j=1}^2h_{jx}(t)Q_{ji}(t)-f_x(t)p_i(t)+l_{ix}(t)\big]dt-k_i(t)dW(t),\\
p_i(0)=&-\gamma_y(y(0)),\\
q_i(T)=&-g_x(x(T))p_i(T)+\Phi_{ix}(x(T))\quad(i=1,2),
\end{aligned}
\right.
\end{equation}
and
\begin{equation}
\left\{
\begin{aligned}
-dP_{i}(t)=&l_i(t)dt-Q_{i}(t)dW(t)-\sum_{j=1}^2Q_{ji}(t)dW_j(t),\\
P_i(T)=&\Phi_i(x(T))\quad(i=1,2).
\end{aligned}
\right.
\end{equation}
\end{remark}

\begin{remark}
If we set $h(\cdot)$ in the observation equation satisfies $h(t,x(\cdot),v_1(\cdot),v_2(\cdot))=h(t,x(\cdot))$, then the Hamiltanian in theorem \ref{MP} becomes:

\begin{equation}
\tilde{H}_{iv_i}(t)=H_{iv_i}(t)\quad(i=1,2),
\end{equation}
where

\begin{equation}
\begin{aligned}
H_{i}(\cdot)\triangleq& b(t,x,u_1,u_2)q_i(t)+\sigma(t,x,u_1,u_2)k_i(t)+\sum_{j=1}^2[\sigma_j(t,x)k_{ji}(t)+h_j(t,x)Q_{ji}(t)]\\
-&[f(t,x,y,z,z_1,z_2,u_1,u_2)-\sum_{j=1}^2h_j(t,x)z_j(t)]p_i(t)+l_i(t,x,y,z,z_1,z_2,u_1,u_2),\\
\end{aligned}
\end{equation}
and the corresponding adjoint equations (\ref{P}), (\ref{pq}) are unchanged.
\end{remark}

\begin{remark}\label{uncontroled}
As mentioned in Remark \ref{circle}, the main obstacle in the construction of the partially observed system is that if observation variable $Y(\cdot)$ relies on control variable, the control will adapted to a controlled filtration, and the $dY(\cdot)$ part will be affected by convex variational method. Here, if we consider the special case that $h_j(\cdot,x(\cdot),v_1(\cdot),v_2(\cdot))=h_j(\cdot)$ in (\ref{Observe}), then $Y_j(\cdot)$ is uncontrolled process naturally. Thus we don't need to use Girsanov theorem to reconstruct the system. We set $W_j^{v_1,v_2}(\cdot)=W_j(\cdot)$ to be the B.M. under probability measure $\mathbb{P}^{v_1,v_2}=\mathbb{P}$ straightly. In that way, the adjoint process $P(\cdot)$ is needless for the reason that control doesn't affect the observation equation any more.

Thus the Hamiltonion in theorem \ref{MP} becomes the classic form below
\begin{equation}
\begin{aligned}
\mathbb{E}[{H}_{i{v_i}}(t,x,y,z,z_1,z_2,u_1,u_2;q_i,k_i,k_{1i},k_{2i},p_i)|\mathcal{F}_t^i]=0\quad(i=1,2),
\end{aligned}
\end{equation}

where
\begin{equation}
\begin{aligned}
{H}_{i}(t)=&H_{i}(t,x,y,z,z_1,z_2,u_1,u_2;q_i,k_i,k_{1i},k_{2i},p_i)\\
=&b(t,x,u_1,u_2)q_i(t)+\sigma(t,x,u_1,u_2)k_i(t)+\sum_{j=1}^2\sigma_j(t,x,u_1,u_2)k_{ji}(t)\\
-&[f(t,x,y,z,z_1,z_2,u_1,u_2)-\sum_{j=1}^2h_j(t)z_j(t)]p_i(t)+l_i(t,x,y,z,z_1,z_2,u_1,u_2),\\
\end{aligned}
\end{equation}

with adjoint process satisfies
\begin{equation}\label{adj unc}
\left\{
\begin{aligned}
dp_i(t)=&-H_{iy}(t)dt-H_{iz}(t)dW(t)-\sum_{j=1}^2H_{iz_j}(t)dW_j(t),\\
-dq_i(t)=&H_{ix}(t)dt-k_i(t)dW(t)-\sum_{j=1}^2k_{ji}(t)dW_j(t),\\
p_i(0)=&-\gamma_y(y(0)),\\
q_i(T)=&-g_x(x(T))p_i(T)+\Phi_{ix}(x(T))\quad(i=1,2),
\end{aligned}
\right.
\end{equation}

and
\begin{equation}
\left\{
\begin{aligned}
-dP_{i}(t)=&l_i(t)dt-Q_{i}(t)dW(t)-\sum_{j=1}^2Q_{ji}(t)dW_j(t),\\
P_i(T)=&\Phi_i(x(T))\quad(i=1,2).
\end{aligned}
\right.
\end{equation}
\end{remark}

\subsection{A sufficient condition (verification theorem)}

Here we establish the sufficient condition of the case when the observation process is not affected by the control process. Just as Remark \ref{uncontroled}, we suppose $h(t,x(t),v_1(t),v_2(t))=h(t,x(t))$ and $(W(\cdot),W_1(\cdot),W_2(\cdot))$ is a standard B.M. on probability space $(\Omega,\mathcal{F},\{\mathcal{F}_t\}_{t\geq 0}, \mathbb{P})$.

\begin{theorem}
Suppose hypothesis (H1) (H2) hold, and the adjoint equation (\ref{adj unc}) admits a solution $(p_i(\cdot),q_i(\cdot),k_i(\cdot)$,$k_{1i}(\cdot),k_{2i}(\cdot))\in\mathbb{L}_{\mathcal{F}}^2(0,T;\mathbb{R}^5)$ for $i=1,2.$

Suppose
\begin{equation}
\begin{aligned}
\mathbb{E}[{H}_1(t)|\mathcal{F}_t^1]=\min\limits_{v_1\in U_1}\mathbb{E}[{H}_1^{v_1}(t)|\mathcal{F}_t^1],\\
\mathbb{E}[{H}_2(t)|\mathcal{F}_t^2]=\min\limits_{v_2\in U_2}\mathbb{E}[{H}_2^{v_2}(t)|\mathcal{F}_t^2],
\end{aligned}
\end{equation}

where
\begin{equation}
\begin{aligned}
{H}_{i}(t)=& H_{i}(t,x,y,z,z_1,z_2,u_1,u_2;q_i,k_i,k_{1i},k_{2i},p_i)\quad(i=1,2),\\
{H}_{1}^{v_1}(t)=& H_{1}(t,x,y,z,z_1,z_2,v_1,u_2;q_1,k_1,k_{11},k_{21},p_1),\\
{H}_{2}^{v_2}(t)=& H_{2}(t,x,y,z,z_1,z_2,u_1,v_2;q_2,k_2,k_{12},k_{22},p_2).
\end{aligned}
\end{equation}

Suppose that $\mathbb{E}[{H}_{iv_i}^{v_i}(t)|\mathcal{F}_t^i]$ is continuous at $v_i=u_i(t)\quad(i=1,2).$

Suppose
\begin{equation}
\begin{aligned}
(t,x,y,z,z_1,z_2,v_i)\mapsto& H_i^{v_i}(t)\quad(i=1,2),\\
x\mapsto& g(x)\\
x\mapsto& \Phi_i(x)\ \quad(i=1,2),\\
x\mapsto& \gamma_i(y)\ \ \quad(i=1,2)
\end{aligned}
\end{equation}
are convex functions respectively. Then, $(u_1(\cdot),u_2(\cdot))$ is the open-loop Nash equilibrium point.
\end{theorem}

\begin{proof}
We only prove the case of $i=1$. For $\forall v_1(\cdot)\in \mathcal{U}_1$, we have
\begin{equation}
J_1(v_1(\cdot),u_2(\cdot))-J_1(u_1(\cdot),u_2(\cdot))=A+B+C,
\end{equation}

with

\begin{equation}
\begin{aligned}
A=&\mathbb{E}\int_0^T[l_1^{v_1}(t)-l_1(t)]dt,\\
B=&\gamma_1(y^{v_1}(0))-\gamma_1(y(0)),\\
C=&\mathbb{E}[\Phi_1(x^{v_1}(T))-\Phi_1(x(T))].
\end{aligned}
\end{equation}

Due to $\gamma_1(y)$ is convex on $y$,

\begin{equation}
\begin{aligned}
B\geq\gamma_{1y}(y(0))(y^{v_1}(0)-y(0)).
\end{aligned}
\end{equation}

Using It\^o's formula to $p_1(\cdot)(y^{v_1}(0)-y(0))$,

\begin{equation}
\begin{aligned}
B\geq\mathbb{E}\int_0^T\Big[-&p_1(t)\Big(f^{v_1}(t)-f(t)-\sum_{j=1}^2(z_j^{v_1}(t)-z_j(t))h_j(t)\Big)-\big(y^{v_1}(t)-y(t)\big)H_{1y}(t)-\big(z^{v_1}(t)-z(t)\big)H_{1z}(t)\\
-&\sum_{j=1}^2\big(z_j^{v_1}(t)-z_j(t)\big)H_{1z_j}(t)\Big]dt-p_1(T)\Big(g(x^{v_1}(T)-g(x(T))\Big).
\end{aligned}
\end{equation}

Due to $\Phi_{1}(x)$ is convex on $x$,

\begin{equation}
\begin{aligned}
C\geq\Phi_{1x}(x(T))(x^{v_1}(T)-x(T)).
\end{aligned}
\end{equation}

Using It\^o's formula to $q_1(\cdot)(x^{v_1}(T)-x(T))$,

\begin{equation}
\begin{aligned}
C\geq\mathbb{E}\int_0^T\Big[&q_1(t)(b^{v_1}(t)-b(t))-(x^{v_1}(t)-x(t))H_{1x}(t)+k_1(t)(\sigma^{v_1}(t)-\sigma(t))\\
+&\sum_{j=1}^2k_{j1}(t)(\sigma_j^{v_1}(t)-\sigma_j(t))\Big]dt.
\end{aligned}
\end{equation}

Moreover, we have

\begin{equation}
\begin{aligned}
A=\mathbb{E}&\int_0^T\big[H_1^{v_1}(t)-H_1(t)\big]dt-\mathbb{E}\int_0^T\bigg[\big(b^{v_1}(t)-b(t)\big)q_1(t)+\big(\sigma^{v_1}(t)-\sigma(t)\big)k_1(t)\\
+&\sum_{j=1}^2\big(\sigma_j^{v_1}(t)-\sigma_j(t)\big)k_{j1}(t)-\big[f^{v_1}(t)-f(t)-\sum_{j=1}^2\big(z_j^{v_1}(t)-z_j(t)\big)h_j(t)\big]p_1(t)\bigg]dt.
\end{aligned}
\end{equation}

From $A,B,C$,

\begin{equation}
\begin{aligned}
J_1(v_1(\cdot),u_2(\cdot))-&J_1(u_1(\cdot),u_2(\cdot))\geq \mathbb{E}\int_0^T\bigg[(H_1^{v_1}(t)-H_1(t))-(x^{v_1}(t)-x(t))H_{1x}(t)\\
-&(y^{v_1}(t)-y(t))H_{1y}(t)-(z^{v_1}(t)-z(t))H_{1z}(t)-\sum_{j=1}^2(z_j^{v_1}(t)-z_j(t))H_{1z_j}(t)\bigg]dt.
\end{aligned}
\end{equation}

Due to $(t,x,y,z,z_1,z_2,v_1)\mapsto H_1^{v_1}(t)$ is convex,

\begin{equation}
\begin{aligned}
J_1(v_1(\cdot),u_2(\cdot))-J_1(u_1(\cdot),u_2(\cdot))\geq& \mathbb{E}\int_0^TH_{1v_1}(t)(v_1(t)-u_1(t))dt\\
=&\mathbb{E}\int_0^T\mathbb{E}\Big[H_{1v_1}(t)\big(v_1(t)-u_1(t)\big)|\mathcal{F}_t^1\Big]dt.
\end{aligned}
\end{equation}

From the assumption $v_1\mapsto \mathbb{E}[H_1^{v_1}(t)|\mathcal{F}_t^1]$ is minimal at $v_1=u_1(t)$ for $\forall t\in[0,T]$ and $H_{1v_1}^{v_1}(t)$ is continuous on $v_1$, then we have

\begin{equation}
\mathbb{E}\Big[H_{1v_1}(t)(v_1(t)-u_1(t))|\mathcal{F}_t^1\Big]=\Big(\frac{\partial}{\partial v_1}\mathbb{E}\big[H_1(t)|\mathcal{F}_t^1\big]\Big)(v_1(t)-u_1(t))\geq 0.
\end{equation}

Thus, it implies that
\begin{equation}
J_1(u_1(\cdot),u_2(\cdot))=\min\limits_{v_1(\cdot)\in\mathcal{U}_1}J_1(v_1(\cdot),u_2(\cdot)).
\end{equation}

Similarly, we can prove the case when $i=2$,
\begin{equation}
J_2(u_1(\cdot),u_2(\cdot))=\min\limits_{v_2(\cdot)\in\mathcal{U}_2}J_1(u_1(\cdot),v_2(\cdot)).
\end{equation}
\end{proof}

\section{An example in finance}
In this section, we consider a realistic investment problem in our financial market. We can solve it by using the necessary and sufficient condition we derived in section 2, and give the related Nash equilibrium strategy explicitly.

We assume that there are $n+1$ assets can be continuously traded in financial market.

1 bond
\begin{equation}
\left\{
\begin{aligned}
dB(t)=&r(t)B(t)dt,\\
B(0)=&1.
\end{aligned}
\right.
\end{equation}

$n$ stocks
\begin{equation}
\left\{
\begin{aligned}
dS_i(t)=&\mu_i(t)S_i(t)dt+\sum_{j=1}^n\sigma_{ij}(t)S_i(t)dW_j(t),\\
S_i(0)=&1\quad (i=1,2\ldots n),
\end{aligned}
\right.
\end{equation}
where $W(\cdot)=(W_1(\cdot),\ldots,W_n(\cdot))^\tau$ is $n$-dimensional standard B.M. defined on probability space $(\Omega,\mathcal{F},\{\mathcal{F}_t\}_{t\geq 0}$,
$\mathbb{P})$. $\mu(\cdot)=(\mu_1(\cdot),\ldots,\mu_n(\cdot))^\tau$ is appreciation rate of the stock process. The $n\times n$ matrix valued process $\Sigma(t)=(\sigma_{ij}(t))$ is the volatility coefficients of the stock process. We set $S(\cdot)=(S_1(\cdot),\ldots,S_n(\cdot))^\tau$.

We make the following assumptions.

\textbf{Hypothesis (H3)}.
$\mu(\cdot)$ is $\mathcal{F}_t$-adapted bounded process, $r(\cdot)$ and $\sigma_{ij}(\cdot)$ are deterministic bounded coefficients. $\Sigma(\cdot)$ has full rank for $\forall t\in [0,T]$, and the inverse matrix $\Sigma(\cdot)^{-1}$ is bounded.

We suppose that there is a company hires two managers. Each of them observes few stocks from $S(\cdot)$.

\begin{equation}
\begin{aligned}
\text{manager 1:}\ \ dS^1_i(t)=&\mu_i^1(t)S_i^1(t)dt+\sum_{j=1}^{n_1}\sigma_{ij}^1(t)S_i^1(t)dW_j^1(t)\quad(i=1,\ldots,n_1), \\
\text{manager 2:}\ \ dS^2_i(t)=&\mu_i^2(t)S_i^2(t)dt+\sum_{j=1}^{n_2}\sigma_{ij}^2(t)S_i^2(t)dW_j^2(t)\quad(i=1,\ldots,n_2),
\end{aligned}
\end{equation}
where $n_k$ is the dimension of observed stocks $S^k(\cdot)=(S^k_1(\cdot),\ldots,S^k_{n_k}(\cdot))^\tau$ for $k=1,2$,  which are parts of real stock process $S(\cdot)$ corresponding to the two managers respectively. Thus, we denote the rest of both unobservable part of stock process as $S^0(\cdot)=(S^0_1(\cdot),\ldots,S^0_{n_0}(\cdot))^\tau$, which can be also invested by company. Here $W^k(\cdot)=(W_1^k(\cdot),\ldots,W_{n_k}^k(\cdot))^\tau,k=0,1,2$ are the corresponding mutually independent $n_k$-dimensional Brownian motions (see Theorem 3.1 in Xiong and Zhou\cite{XiongZhou07}). For calculation simplicity, we might as well suppose there is no common B.M. among the vectors $W^k(\cdot),k=0,1,2$ and no common observed stock between two managers. If not, it will not cause any trouble during calculation but more redundant to be represented. In that way, we have $n_0+n_1+n_2=n$.  We set $\mu^k(\cdot)=(\mu_1^k(\cdot),\ldots,\mu^k_{n_k}(\cdot))^\tau,k=0,1,2$ and $\Sigma^k(\cdot)=(\sigma_{ij}^k(\cdot))\ i,j=1,\ldots,n_k,k=0,1,2$ to be the corresponding appreciation rate and volatility of stock process $S^k(\cdot)$.  Further, we make the assumption naturally that the appreciation rate $\mu^k(\cdot)$ is unobservable for $k=0,1,2$. We denote $a_{ij}^k(\cdot)=\sum_{l=1}^{n_k}\sigma_{il}^k(\cdot)\sigma_{jl}^k(\cdot), i,j=1,\ldots,n_k, k=0,1,2.$, $A^k(\cdot)=(a_{11}^k(\cdot),\ldots,a_{n_{k}n_{k}}^k(\cdot))^\tau$,

Now we set

\begin{equation}
dY^k(\cdot)=\Sigma^k(\cdot)^{-1}d\log S^k(\cdot)\quad(k=1,2).
\end{equation}

By using It\^o's formula, our observation equation turns to
\begin{equation}\label{observe}
\begin{aligned}
dY^1(t)=&\eta^1(t)dt+dW^1(t),
\end{aligned}
\end{equation}

\begin{equation}\label{observe2}
\begin{aligned}
dY^2(t)=&\eta^2(t)dt+dW^2(t),
\end{aligned}
\end{equation}
where $Y^k(\cdot)=(Y_1^k(\cdot),\ldots,Y_{n_k}^k(\cdot))^\tau\quad(k=1,2).$

and
\begin{equation}\label{eta}
\eta^k(\cdot)\triangleq \Sigma^k(\cdot)^{-1}(\mu^k(t)-\frac{1}{2}A^k(\cdot))\quad(k=1,2).
\end{equation}

Let
\begin{equation}
\mathcal{F}_t^k=\sigma\{B(s),Y^k(s);0\leq s\leq t\}\quad(k=1,2)
\end{equation}
be the available filtration to each manager. In that case, $r(\cdot),\Sigma(\cdot)$ are all completely observable, while $\mu(\cdot)$ is unobservable.

Here we assume that the $n_k$-dimensional drift process $\mu^k(\cdot),k=1,2$ of each observation equation is the solution of the following stochastic equation 

\begin{equation}\label{mu}
\left\{
\begin{aligned}
d\mu^k(t)=&\theta^k(\delta^k-\mu^k(t))dt+\zeta^kd\bar{W}^k(t),\\
\mu^k(0)=&1\quad (k=1,2),
\end{aligned}
\right.
\end{equation}
where  $\bar{W}^k(\cdot)=(\bar{W}^k_1(\cdot),\ldots,\bar{W}^k_{n_k}(\cdot)),k=1,2$ are Brownian motions with respect to $(\Omega,\mathcal{F},\{\mathcal{F}_t\}_{t\geq 0},\mathbb{P})$, independent of $W^k(\cdot),k=1,2$ under $\mathbb{P}$. $\theta^k$ is the $n_k\times n_k$ diagonal matrix with the $i$-th component $\theta^k_i$, $\delta^k=(\delta^k_1,\ldots,\delta^k_{n_k})^\tau$, $\zeta^k=(\zeta^k_{ij}), i,j=1,\ldots,n_k$ are $n_k\times n_k$ matrix for $k=1,2$. We also suppose that $\theta^k_i$, $\delta^k_1$, $\zeta_{ij}$ are all positive constants for $ i,j=1,\ldots,n_k,k=1,2$. Then $\mu^k(\cdot),k=1,2.$ is the $n_k$-dimensional Ornstein-Uhlenbeck process with mean reverting drift.

We assume the company plans to obtain a terminal wealth $\xi$, which is a $\mathcal{F}_t$-adapted non-negative random variable satisfying $\mathbb{E}|\xi|^2<\infty$. Now the whole wealth of the company is denoted by $y(\cdot)$. The first manager invests $\pi^1_i(t)$ wealth in stock $S^1_i(t)(i=1,\ldots,n_1)$ he observed, and the second manager invests $\pi^2_i(t)$ wealth in stock $S^2_i(t)(i=1,\ldots,n_2)$ he focused on. We suppose that there are $\pi^0_i(t)$ wealth invested  by company in unobservable stocks  $S^0_i(t)(i=1,\ldots,n_0)$ of both managers. So the rest $y(t)-\sum_{k=0}^2\sum_{i=1}^{n_k}\pi_i^k(t)$ wealth invested in bond. Thus we can establish the wealth equation as

\begin{equation}\label{wealth}
\left\{
\begin{aligned}
dy(t)=&\big[r(r)y(t)+\sum_{k=0}^2\sum_{i=1}^{n_k}(\mu_i^k(t)-r(t))\pi_i^k(t)+I_1(t)+I_2(t)\big]dt+\sum_{k=0}^2\sum_{i=1}^{n_k}\sum_{j=1}^{n_k}\pi_i^k(t)\sigma_{ij}^k(t)dW_j^k(t),\\
y(T)=&\xi,
\end{aligned}
\right.
\end{equation}
where $I_1(\cdot),I_2(\cdot)$ are represented to the instantaneous capital injection of each manager to guarantee the terminal wealth of the company.

For each of them, their mission is to use the minimal capital injection wealth and the minimum start-up capital to make sure the company reach the ultimate wealth. Meanwhile, any one of them has his own utility on their injection process. The more capital he uses, the danger he undertakes. So we can define the related utility function for each manager.

\begin{equation}\label{cost}
J_i(I_1(\cdot),I_2(\cdot))=\mathbb{E}\int_0^T[L_ie^{-\beta t}I_i^2(t)dt+M_iy(0)]\quad(i=1,2),
\end{equation}
where $L_i,M_i$ are two positive constants, $\beta$ is the discount rate. We define the risk-seeking running cost as a criteria for injection utility. To attain the terminal wealth, they want to minimize both of the injection utility and the start-up wealth value. That is

\begin{equation}
\left\{
\begin{aligned}
J_1(\bar{I}_1(\cdot),\bar{I}_2(\cdot))=\min\limits_{I_1(\cdot)\in\mathcal{I}_1}J_1(I_1(\cdot),\bar{I}_2(\cdot)),\\
J_2(\bar{I}_1(\cdot),\bar{I}_2(\cdot))=\min\limits_{I_2(\cdot)\in\mathcal{I}_2}J_2(\bar{I}_1(\cdot),I_2(\cdot)),
\end{aligned}
\right.
\end{equation}
where we define
\begin{equation}
\mathcal{I}_i=\{I_i(\cdot)\in L^2_{\mathcal{F}_\cdot^i}(0,T;\mathbb{R});I_i(t)\geq 0,t\in [0,T])\}\quad(i=1,2).
\end{equation}

We regard $(\bar{I}_1(\cdot),\bar{I}_2(\cdot))$ as the open-loop Nash equilibrium strategy of this investment problem.

The wealth equation (\ref{wealth}) is a backward case. We can denote $\pi^k(t)=(\pi_i^k(t),\ldots,\pi_{n_k}^k(t))^\tau$, $z^k(t)=(z_i^k(t),\ldots,z_{n_k}^k(t))=\pi^k(t)^\tau\Sigma^k(t)\quad(k=0,1,2).$

Therefore, our wealth equation turns into

\begin{equation}\label{wealth new}
\left\{
\begin{aligned}
dy(t)=&\big[r(t)y(t)+\sum_{k=0}^2b^k(t)^\tau z^k(t)^\tau+I_1(t)+I_2(t)\big]dt+\sum_{k=0}^2z^k(t)dW^k(t),\\
y(T)=&\xi,
\end{aligned}
\right.
\end{equation}
where
\begin{equation}\label{b}
b^k(t)=\Sigma^k(t)^{-1}(\mu^k(t)-r(t))\quad(k=0,1,2).
\end{equation}

Form (\ref{wealth new}) and (\ref{cost}), we use the maximum principle derived in section 2.

The Hamiltonian functions are
\begin{equation}
H_i(t,y,z^0,z^1,z^2,I_1,I_2;p_i)=\Big(r(t)y(t)+\sum_{k=0}^2b^k(t)^\tau z^k(t)^\tau+I_1(t)+I_2(t)\Big)p_i(t)+L_ie^{-\beta t}I_i^2(t),
\end{equation}
for $i=1,2.$

The adjoint process $p_i(\cdot)$ satisfies

\begin{equation}
\left\{
\begin{aligned}
dp_i(t)=&-r(t)p_i(t)dt-\sum_{k=0}^2b^k(t)^\tau p_i(t)dW^k(t),\\
dp_i(0)=&-M_i\quad(i=1,2).
\end{aligned}
\right.
\end{equation}

From the necessary condition, we can find a candidate open-loop Nash equilibrium point

\begin{equation}
\left\{
\begin{aligned}
\bar{I}_1(t)=&-\frac{1}{2}e^{\beta t}L_1^{-1}\widehat{p_1}(t),\\
\bar{I}_2(t)=&-\frac{1}{2}e^{\beta t}L_2^{-1}\widetilde{p_2}(t),
\end{aligned}
\right.
\end{equation}
where we set $\widehat{\phi}(t)=\mathbb{E}[\phi(t)|\mathcal{F}_t^1]$,\ \  $\widetilde{\psi}(t)=\mathbb{E}[\psi(t)|\mathcal{F}_t^2]$ for $\forall\phi(\cdot),\psi(\cdot)\in\mathcal{F}_\cdot$.

Now focusing on the $i=1$ case, from observation equation (\ref{observe}) and the Kushner-FKK equation in Xiong \cite{Xiong08}, we have

\begin{equation}\label{filtering adjoint}
\left\{
\begin{aligned}
d\widehat{p}_1(t)=&-r(t)\widehat{p}_1(t)dt+\Big[-\widehat{b^1(t)^\tau p_1(t)}+\widehat{\eta^1(t)^\tau p_1(t)}-\widehat{\eta^1}(t)^\tau\widehat{p}_1(t)\Big]d\widehat{W}^1(t),\\
d\widehat{p}_1(0)=&-M_1,
\end{aligned}
\right.
\end{equation}
where the innovation process $\widehat{W}^1(\cdot)$ satisfying
\begin{equation}
\widehat{W}^1(t)=Y^1(t)-\int_0^t\widehat{\eta^1}(s)ds
\end{equation}
is a $\mathcal{F}_t^1$-Brownian motion under probability measure $\mathbb{P}$.

From $\eta^1(\cdot)$ in (\ref{eta}), $b^1(t)$ in (\ref{b}), we find that
\begin{equation}\label{etatob}
\widehat{\eta^1(t)p_1(t)}-\widehat{\eta^1}(t)\widehat{p}_1(t)=\Sigma^1(t)^{-1}\big(\widehat{\mu^1(t)p_1(t)}-\widehat{\mu^1}(t)\widehat{p}_1(t)\big)=\widehat{b^1(t)p_1(t)}-\widehat{b^1}(t)\widehat{p}_1(t).
\end{equation}

Substituting (\ref{etatob}) into (\ref{filtering adjoint}), we get
\begin{equation}
\left\{
\begin{aligned}
d\widehat{p}_1(t)=&-r(t)\widehat{p}_1(t)dt-\widehat{b^1}(t)^\tau\widehat{p}_1(t)d\widehat{W}^1(t),\\
d\widehat{p}_1(0)=&-M_1.\qquad
\end{aligned}
\right.
\end{equation}

Thus
\begin{equation}\label{phat}
\widehat{p}_1(t)=-M_1\exp\{\int_0^t[-r(s)-\frac{1}{2}\widehat{b^1}(s)^2]ds-\int_0^t\widehat{b^1}(s)d\widehat{W}^1(s)\}.
\end{equation}

From equation (\ref{observe}), (\ref{mu}) and Theorem 8.1 in  \cite{Xiong08}, we have

\begin{equation}\label{muhat}
\left\{
\begin{aligned}
d\widehat{\mu^1}(t)=&\theta^1(\delta^1-\widehat{\mu^1}(t))dt+(P_1(t)(\Sigma^1(t)^{-1})^\tau)d\widehat{W}^1(t),\\
\mu^k(0)=&I\quad (k=1,2),
\end{aligned}
\right.
\end{equation}
and
$P_1(\cdot)=\mathbb{E}\big[\big(\mu^1(t)-\widehat{\mu^1}(t)\big)\big(\mu^1(t)-\widehat{\mu^1}(t)\big)^\tau\big]=\mathbb{E}\big[\big(\mu^1(t)-\widehat{\mu^1}(t)\big)\big(\mu^1(t)-\widehat{\mu^1}(t)\big)^\tau|\mathcal{F}_t^1]$ satisfies

\begin{equation}\label{P}
\left\{
\begin{aligned}
&\dot{P_1}(t)+2\theta^1P_1(t)+P_1(t)(\Sigma^1(t)^{-1})^\tau\Sigma^1(t)^{-1}P_1(t)-\zeta^1(\zeta^1)^\tau=0,\\
&P_1(0)=I.
\end{aligned}
\right.
\end{equation}

From classic Riccati equation theory, (\ref{P}) exists a unique solution. Then, we obtain the unique expression of $\widehat{b^1}(\cdot)$ in (\ref{phat}).

Similarly, we can obtain $\widetilde{p}_2(\cdot)$ satisfies
\begin{equation}
\left\{
\begin{aligned}
d\widetilde{p}_2(t)=&-r(t)\widetilde{p}_2(t)dt-\widetilde{b^2}(t)^\tau\widetilde{p}_2(t)d\widetilde{W}^2(t),\\
d\widetilde{p}_2(0)=&-M_1,
\end{aligned}
\right.
\end{equation}

where
\begin{equation}
\widetilde{W}^2(t)=Y^2(t)-\int_0^t\widetilde{\eta^2}(s)ds.
\end{equation}

Thus
\begin{equation}\label{phat2}
\widetilde{p}_2(t)=-M_1\exp\{\int_0^t[-r(s)-\frac{1}{2}\widetilde{b^2}(s)^2]ds-\int_0^t\widetilde{b^2}(s)d\widetilde{W}^2(s)\}.
\end{equation}

From equation (\ref{observe2}), (\ref{mu}) and Theorem 8.1 in  \cite{Xiong08}, we have

\begin{equation}\label{muhat}
\left\{
\begin{aligned}
d\widetilde{\mu^2}(t)=&\theta^2(\delta^2-\widetilde{\mu^2}(t))dt+(P_2(t)(\Sigma^2(t)^{-1})^\tau)d\widetilde{W}^2(t),\\
\mu^k(0)=&I\quad (k=1,2),
\end{aligned}
\right.
\end{equation}
and
$P_2(\cdot)=\mathbb{E}\big[\big(\mu^2(t)-\widetilde{\mu^2}(t)\big)\big(\mu^2(t)-\widetilde{\mu^2}(t)\big)^\tau\big]=\mathbb{E}\big[\big(\mu^2(t)-\widetilde{\mu^2}(t)\big)\big(\mu^2(t)-\widetilde{\mu^2}(t)\big)^\tau|\mathcal{F}_t^2]$ satisfies

\begin{equation}\label{P2}
\left\{
\begin{aligned}
&\dot{P}_2(t)+2\theta^2P_2(t)+P_2(t)(\Sigma^2(t)^{-1})^\tau\Sigma^2(t)^{-1}P_2(t)-\zeta^2(\zeta^2)^\tau=0,\\
&P_2(0)=I.
\end{aligned}
\right.
\end{equation}

From classic Riccati equation theory, (\ref{P2}) exists a unique solution. Then, we obtain the unique expression of $\widetilde{b^2}(\cdot)$ in (\ref{phat2}).

Finally, from the linearity of state processes and convexity of cost functions as well as the sufficient condition we discussed above, we know that $(\bar{I}_1(\cdot),\bar{I}_2(\cdot))$ is the open-loop Nash equilibrium strategy satisfies

\begin{equation}
\left\{
\begin{aligned}
\bar{I}_1=&\frac{1}{2}e^{\beta t}L_1^{-1}M_1\exp\{\int_0^t[-r(s)-\frac{1}{2}\widehat{b^1}(s)^2]ds-\int_0^t\widehat{b^1}(s)d\widehat{W}^1(s)\},\\
\bar{I}_2=&\frac{1}{2}e^{\beta t}L_2^{-1}M_1\exp\{\int_0^t[-r(s)-\frac{1}{2}\widetilde{b^2}(s)^2]ds-\int_0^t\widetilde{b^2}(s)d\widetilde{W}^2(s)\}.
\end{aligned}
\right.
\end{equation}

\end{document}